\documentclass[12pt]{article}
\usepackage{amssymb}
\usepackage{amsmath}
\usepackage{amsthm}
\usepackage{cite}
\usepackage{calc}
\usepackage[font={scriptsize}]{caption}
\usepackage{enumitem}
\setlist{itemindent=0pt,partopsep=0pt,topsep=0pt}
\usepackage{mathrsfs}
\usepackage{tikz}
  \usetikzlibrary{arrows}
  \usetikzlibrary{calc}
  \usetikzlibrary{decorations.markings}
  \usetikzlibrary{positioning}
\usepackage{url}
\allowdisplaybreaks[1]
\newcommand{\C}{\mathbb{C}}
\newcommand{\CP}{\mathbb{CP}}

\renewcommand{\P}{\mathbb{P}}

\newcommand{\R}{\mathbb{R}}

\newcommand{\PO}{\operatorname{PO}}

\newcommand{\PU}{\operatorname{PU}}

\newcommand{\SU}{\operatorname{SU}}
\newcommand{\U}{\operatorname{U}}
\renewcommand\bar[1]{\overline{#1}}

\newcommand{\tr}{\operatorname{Tr}}
\newcommand\Id{\operatorname{Id}}
\newcommand\mbf[1]{\mathbf{#1}}

\theoremstyle{plain}
\newtheorem{thm}{Theorem}
\newtheorem*{thm*}{Theorem}
\newtheorem{prop}[thm]{Proposition}
\newtheorem{lem}[thm]{Lemma}
\newtheorem{cor}[thm]{Corollary}
\theoremstyle{definition}
\newtheorem*{defn}{Definition}
\theoremstyle{remark}
\newtheorem*{rmk}{Remark}

\title{Non-arithmetic hybrid lattices in $\operatorname{PU}(2,1)$}
\author{Joseph Wells \\ Arizona State University}
\date{\today}
\begin{document}
\maketitle

\begin{abstract}
  We explore hybrid subgroups of certain non-arithmetic lattices in $\PU(2,1)$. We show that all of Mostow's lattices are virtually hybrids; moreover, we show that some of these non-arithmetic lattices are hybrids of two non-commensurable arithmetic lattices in $\PU(1,1)$.
\end{abstract}

\section{Introduction}

One key notion in the study of lattices in a semisimple Lie group $G$ is that of \emph{arithmeticity} (which we will not define here; see \cite{Morris} for a standard reference). When $G$ arises as the isometry group of a symmetric space $X$ of non-compact type, the combined work of Margulis \cite{Margulis}, Gromov--Schoen \cite{GromovSchoen}, and Corlette \cite{Corlette} imply that non-arithmetic lattices only exist when $X=\mbf{H}_{\R}^{n}$ or $\mbf{H}_{\C}^{n}$ (real and complex hyperbolic space, respectively); equivalently, up to finite index, when $G=\PO(n,1)$ or $\PU(n,1)$. Due to their exceptional nature, it has been a major challenge to find and understand non-arithmetic lattices in these Lie groups.

Given two arithmetic lattices $\Gamma_1, \Gamma_2$ in $\PO(n,1)$ with common sublattice $\Gamma_{1,2} \leq PO(n-1,1)$, Gromov and Piatestki-Shapiro showed in \cite{GPS} that one can produce a new ``hybrid'' lattice $\Gamma$ in $\PO(n,1)$ by way of a technique that they call ``interbreeding'' or ``hybridization''. In particular, when $\Gamma_1$ and $\Gamma_2$ are not commensurable, $\Gamma$ is shown to be non-arithmetic. It has been asked whether an analogous technique can exist for lattices in $\PU(n,1)$. 

Hunt proposed one possible analog (see the references contained in \cite{PaupertHybrids}) where one starts with two arithmetic lattices $\Gamma_1, \Gamma_2$ in $\PU(n,1)$ and embeddings $\iota_i: \PU(n,1) \hookrightarrow \PU(n+1,1)$ such that (1) $\iota_1(\Gamma_1)$ and $\iota_2(\Gamma_2)$ stabilize totally geodesic complex hypersurfaces in $\mbf{H}_{\C}^{n+1}$, (2) these hypersurfaces are orthogonal to one another, and (3) $\iota_1(\Gamma_1) \cap \iota_2(\Gamma_2)$ is a lattice in $\PU(n-1,1)$. The hybrid subgroup is then $H(\Gamma_1, \Gamma_2) := \langle \iota_1(\Gamma_1), \iota_2(\Gamma_2) \rangle$.

In \cite{PaupertHybrids} Paupert produces an infinite family of hybrids that are non-discrete. In \cite{PaupertWells} Paupert and the author used the same hybridization technique to produce both arithmetic lattices and thin subgroups in the Picard modular groups. In this note, we explore a more general hybrid construction in the context of the lattices $\Gamma(p,t) \subset \PU(2,1)$ originally produced by Mostow in \cite{Mostow} (see Section \ref{sec:mostow} for explanation of notation). We obtain the following results:

\begin{thm*}
  \begin{enumerate}
  \item All of Mostow's lattices $\Gamma(p,t)$ are virtually hybrids.
  \item The non-arithmetic lattices $\Gamma(4,1/12)$, $\Gamma(5,1/5)$, and $\Gamma(5,11/30)$ are virtually hybrids of two non-commensurable arithmetic lattices in $\PU(1,1)$.
  \end{enumerate}
\end{thm*}

The second part of this theorem highlights the similarity of these hybrids and those hybrids of Gromov--Piatetski-Shapiro, specifically in that the hybridization procedure can produce a non-arithmetic lattice from two noncommensurable arithmetic lattices. The author would like to thank Julien Paupert for many insightful discussions and suggested edits.

\section{Complex hyperbolic geometry and hybrids}
We give a brief overview of relevant definitions in complex hyperbolic geometry; the reader can see \cite{Goldman} for a standard source.

Let $H$ be a Hermitian matrix of signature $(n,1)$ and let $\C^{n,1}$ denote $\C^{n+1}$ endowed with the Hermitian form $\langle \cdot, \cdot \rangle$ coming from $H$. Let $V_{-}$ denote the set of points $z \in \C^{n,1}$ for which $\langle z, z \rangle < 0$, and let $V_{0}$ denote the set of points for which $\langle z,z \rangle = 0$. Given the usual projectivization map $\P: \C^{n,1} - \{0\} \rightarrow \CP^{n}$, \emph{complex hyperbolic $n$-space} is $\mbf{H}_{\C}^{n} = \P(V_{-})$ with distance $d$ coming from the Bergman metric
\begin{align*}
  \cosh^{2}\frac{1}{2}d(\P(x),\P(y)) = \frac{|\langle x,y \rangle|^2}{\langle x,x \rangle \langle y,y \rangle}
\end{align*}
The ideal boundary $\partial_{\infty} \mbf{H}_{\C}^{n}$ is then identified with $\P(V_0)$. 

\subsection{Complex hyperbolic isometries}
Let $U(n,1)$ denote the group of unitary matrices preserving $H$. The holomorphic isometry group of $\mbf{H}_{\C}^{n}$ is $\PU(n,1) = \U(n,1)/\U(1)$, and the full isometry group is generated by $\PU(n,1)$ and the antiholomorphic involution $z \mapsto \bar{z}$. Any holomorphic isometry of $\mbf{H}_{\C}^{n}$ is one of the following three types:
\begin{itemize}[noitemsep,topsep=0pt]
\item \emph{elliptic} if it has a fixed point in $\mbf{H}_{\C}^{n}$.
\item \emph{parabolic} if it has exactly one fixed point in the boundary (and no fixed points in $\mbf{H}_{\C}^{n}$).
\item \emph{loxodromic} if it has exactly two fixed points in the boundary (and no fixed points in $\mbf{H}_{\C}^{n}$).
\end{itemize}
Given a vector $v \in \C^{n,1}$ with $\langle v, v \rangle > 0$ and complex number $\zeta$ with unit modulus, the map
\begin{align*}
  R_{v,\zeta}(x): x \mapsto x + (\zeta-1)\frac{\langle x, v \rangle}{\langle v, v \rangle} v
\end{align*}
is an an isometry of $\mbf{H}_{\C}^{n}$ called a \emph{complex reflection}, and its fixed point set $v^{\perp} \subset \mbf{H}_{\C}^{n}$ is a totally geodesic subspace called a \emph{$\C^{n-1}$-plane} (or a \emph{complex line} when $n=2$). We refer to $v$ as a \emph{polar vector} for the subspace $\P(v^{\perp}) \cap \mbf{H}_{\C}^{n}$; abusing notation slightly we will denote such a projective subspace simply by $v^{\perp}$.

\subsection{Complex hyperbolic hybrid construction}

The lack of totally geodesic real hypersurfaces in $\mbf{H}_{\C}^{n}$ presents an issue in finding a suitable complex-hyperbolic analog of the Gromov--Piatetski-Shapiro hybrid groups. Below we present a slightly more general notion of a hybrid group than that originally introduced by Hunt (see \cite{PaupertHybrids} and the references therein).

\begin{defn} Let $\Gamma_1, \Gamma_2 < \PU(n,1)$ be lattices. We define a \emph{hybrid of $\Gamma_1, \Gamma_2$} to be any group $H(\Gamma_1, \Gamma_2)$ generated by discrete subgroups $\Lambda_1, \Lambda_2 < \PU(n+1,1)$ stabilizing totally geodesic hypersurfaces $\Sigma_1, \Sigma_2$ (respectively) such that
  \begin{enumerate}
  \item $\Sigma_1$ and $\Sigma_2$ are orthogonal,
  \item $\Gamma_i = \left.\Lambda_{i}\right|_{\Sigma_i}$, and
  \item $\Lambda_1 \cap \Lambda_2$ is a lattice in $\PU(n-1,1)$.
  \end{enumerate}
\end{defn}

\begin{rmk}
  The groups explored by Paupert and the author in \cite{PaupertHybrids} and \cite{PaupertWells} are still hybrids in this new sense as well.
\end{rmk}

\section{Mostow's lattices}\label{sec:mostow}

In \cite{Mostow}, Mostow constructed the first known non-arithmetic lattices in $\PU(2,1)$ among a family of groups generated by complex reflections. These groups, denoted $\Gamma(p,t)$, are defined as follows: Let $p = 3,4,5$, $t$ be a real number satisfying $|t|<3\!\left(\frac{1}{2}-\frac{1}{p}\right)$, $\alpha = \frac{1}{2\sin(\pi/p)}$, $\varphi = e^{\pi i t/3}$, and $\eta = e^{\pi i/p}$. Define a Hermitian form $\langle x,y \rangle = x^{T} H \bar{y}$ where
\begin{align*}
  H = \begin{pmatrix} 1 & -\alpha \varphi & - \alpha \bar{\varphi} \\ -\alpha \bar{\varphi} & 1 & -\alpha \varphi \\ -\alpha \varphi & -\alpha \bar{\varphi} & 1 \end{pmatrix}.
\end{align*}
For any pair $(p,t)$ as above, the group $\Gamma(p,t)$ is generated by the three complex reflections of order $p$,
\begin{align*}
  R_1 &= \begin{pmatrix}
    \eta^2 & -i \eta \bar{\varphi} & -i \eta \varphi \\
    0 & 1 & 0 \\
    0 & 0 & 1
  \end{pmatrix},
  &
  R_2 &= \begin{pmatrix}
    1 & 0 & 0 \\
    -i \eta \varphi & \eta^2 & - i \eta \bar{\varphi} \\
    0 & 0 & 1
  \end{pmatrix},
  &
  R_3 &= \begin{pmatrix}
    1 & 0 & 0 \\
    0 & 1 & 0 \\
    - i \eta \bar{\varphi} & - i \eta \varphi & \eta^2
  \end{pmatrix},
\end{align*}
and these reflections satisfy the braid relations $R_i R_j R_i = R_j R_i R_j$. The mirror for the reflection $R_i$ is given by $e_i^{\perp}$ where $e_i$ is the standard $i^{\text{th}}$ basis vector. When $|t| < \frac{1}{2} - \frac{1}{p}$, Mostow refers to these groups has having \emph{small phase shift}. We'll similarly refer to $|t|= \frac{1}{2} - \frac{1}{p}$ as having \emph{critical phase shift} and $|t| > \frac{1}{2} - \frac{1}{p}$ as having \emph{large phase shift}. Since the groups $\Gamma(p,t)$ and $\Gamma(p,-t)$ are isomorphic, we restrict our focus to the cases where $t \geq 0$.

\begin{rmk}[Tables 1 and 2 in \cite{Mostow}]\label{rmk:t_value_table}
  For $p=3,4,5$, there are only finitely-many values of $t$ for which $\Gamma(p,t)$ is discrete, and they are given in Table \ref{table:discrete_vals}. If $\Gamma(p,t)$ is discrete, we'll refer to the pair $(p,t)$ as \emph{admissible}.
\end{rmk}

\begin{table}[h]
  \centering
  \begin{tabular}{|l|l|l|l|}
    \hline
    $p$ & $t < 1/2 - 1/p$ & $t=1/2 - 1/p$ & $t > 1/2 - 1/p$ \\
    \hline
    $3$ & $0$,\, $1/30$,\, $1/18$,\, $1/12$,\, $5/42$\, & $1/6$\, & $7/30$,\, $1/3$\,\\
    $4$ & $0$,\, $1/12$,\, $3/20$ & $1/4$\, & $5/12$\, \\
    $5$ & $1/10$,\, $1/5$\, & & $11/30$,\, $7/10$\, \\
    \hline
  \end{tabular}
  \caption{Values of $p$ and $t$ for which $\Gamma(p,t)$ is discrete.}
  \label{table:discrete_vals}
\end{table}

\begin{thm}[Theorem 17.3 in \cite{Mostow}]\label{thm:mostow}
  For each admissible pair $(p,t)$, the group $\Gamma(p,t)$ is a lattice in $\PU(2,1)$, and the following are non-arithmetic: $\Gamma(3,5/42)$, $\Gamma(3,1/12)$, $\Gamma(3,1/30)$, $\Gamma(4,3/20)$, $\Gamma(4,1/12)$, $\Gamma(5,1/5)$, $\Gamma(5,11/30)$.
\end{thm}

Following the notation in \cite{DFP}, we examine closely related groups $\tilde{\Gamma}(p,t) = \langle R_1, J \rangle$ where
\begin{align*}
  J &= \begin{pmatrix}
    0 & 0 & 1 \\
    1 & 0 & 0 \\
    0 & 1 & 0
  \end{pmatrix}.
\end{align*}
$J$ has order $3$ and $R_{i+1} = J R_i J^{-1}$ (where $i=1,2,3$ and indices are taken modulo $3$). It is sufficient to study these groups $\tilde{\Gamma}(p,t)$ due to the following result:

\begin{prop}[Lemma 16.1 in \cite{Mostow}]
  For each admissible pair $(p,t)$, the group $\Gamma(p,t)$ has index dividing $3$ in $\tilde{\Gamma}(p,t)$. The two groups are equal precisely when $k=\frac{1}{2} - \frac{1}{p} - \frac{1}{t}$ and $\ell = \frac{1}{2} - \frac{1}{p} + \frac{1}{t}$ are both integers and $3$ does not divide both $k$ and $\ell$. 
\end{prop}

\section{Hybrids in Mostow's lattices}

As Deraux--Falbel--Paupert show in \cite{DFP}, when $\tilde{\Gamma}(p,t)$ has small phase shift, a fundamental domain for this group can be constructed by coning over two polytopes that intersect in a right-angled hexagon (see Figure \ref{fig:small_phase}) whose walls are determined by the polar vectors $v_{ijk}$. Taking lifts to $\C^{2,1}$ these vectors are given explicitly below:
\begin{align*}
  v_{123} &= \begin{pmatrix}
    - i \eta \bar{\varphi} \\ 1 \\ i \bar{\eta} \varphi
  \end{pmatrix},
  &
  v_{231} &= \begin{pmatrix}
    i \bar{\eta} \varphi \\ - i \eta \bar{\varphi} \\ 1
  \end{pmatrix},
  &
  v_{312} &= \begin{pmatrix}
    1 \\ i \bar{\eta} \varphi \\ - i \eta \bar{\varphi}
  \end{pmatrix},
  \\
  v_{321} &= \begin{pmatrix}
    i \bar{\eta} \bar{\varphi} \\ 1 \\ -i \eta \varphi
  \end{pmatrix},
  &
  v_{132} &= \begin{pmatrix}
    -i \eta \varphi \\ i \bar{\eta} \bar{\varphi} \\ 1
  \end{pmatrix},
  &
  v_{213} &= \begin{pmatrix}
    1 \\ -i \eta \varphi \\ i \bar{\eta} \bar{\varphi}
  \end{pmatrix}.
\end{align*}
Geometrically, $v_{ijk}^{\perp}$ is the mirror for the complex reflection $J^{\pm 1} R_j R_k$ for $k = i \pm 1 \pmod 3$. When $\tilde{\Gamma}(p,t)$ has critical phase shift, the hexagon degenerates into an ideal triangle (see Figure \ref{fig:critical_phase}) and $J R_j R_k$ is parabolic (hence $\tilde{\Gamma}(p,t)$ is non-cocompact). When $\tilde{\Gamma}(p,t)$ has large phase shift, the ideal vertices sit inside $\mbf{H}_{\C}^{2}$ (see Figure \ref{fig:large_phase}) and $J R_j R_k$ is elliptic.

\begin{figure}[h!]
  \centering
  \begin{tikzpicture}[scale=1.25]
  \def\outvrad{2.25}
  \def\hexcentrad{2.73205}
  \def\mlabeldist{0.800199}
  \coordinate (V321) at (-150:\outvrad);
  \coordinate (V312) at (-90:\outvrad);
  \coordinate (V132) at (-30:\outvrad);
  \coordinate (V123) at (30:\outvrad);
  \coordinate (V213) at (90:\outvrad);
  \coordinate (V231) at (150:\outvrad);
  \coordinate (center321) at (-150:\hexcentrad);
  \coordinate (center312) at (-90:\hexcentrad);
  \coordinate (center132) at (-30:\hexcentrad);
  \coordinate (center123) at (30:\hexcentrad);
  \coordinate (center213) at (90:\hexcentrad);
  \coordinate (center231) at (150:\hexcentrad);
  \draw [very thick] (0.,0.) circle (1.93185);
  \draw[shift={(center123)},very thick,color=blue]  plot[domain=3.403392:3.926991,variable=\t]({1.931852*cos(\t r)},{1.931852*sin(\t r)});
  \draw[shift={(center213)},very thick,color=blue,rotate=60]  plot[domain=3.403392:3.926991,variable=\t]({1.931852*cos(\t r)},{1.931852*sin(\t r)});
  \draw[shift={(center231)},very thick,color=blue,rotate=120]  plot[domain=3.403392:3.926991,variable=\t]({1.931852*cos(\t r)},{1.931852*sin(\t r)});
  \draw[shift={(center321)},very thick,color=blue,rotate=180]  plot[domain=3.403392:3.926991,variable=\t]({1.931852*cos(\t r)},{1.931852*sin(\t r)});
  \draw[shift={(center312)},very thick,color=blue,rotate=-120]  plot[domain=3.403392:3.926991,variable=\t]({1.931852*cos(\t r)},{1.931852*sin(\t r)});
  \draw[shift={(center132)},very thick,color=blue,rotate=-60]  plot[domain=3.403392:3.926991,variable=\t]({1.931852*cos(\t r)},{1.931852*sin(\t r)});
  \draw[thin,color=blue] (64:0.935565) -- (60:0.862318) -- (56:0.935565);
  \draw[thin,color=blue,rotate=60] (64:0.935565) -- (60:0.862318) -- (56:0.935565);
  \draw[thin,color=blue,rotate=120] (64:0.935565) -- (60:0.862318) -- (56:0.935565);
  \draw[thin,color=blue,rotate=180] (64:0.935565) -- (60:0.862318) -- (56:0.935565);
  \draw[thin,color=blue,rotate=-120] (64:0.935565) -- (60:0.862318) -- (56:0.935565);
  \draw[thin,color=blue,rotate=-60] (64:0.935565) -- (60:0.862318) -- (56:0.935565);
  \draw[very thick,color=red] (150:0.800199) -- (-30:0.800199);
  \draw[very thick,color=red] (30:0.800199) -- (-150:0.800199);
  \draw[very thick,color=red] (90:0.800199) -- (-90:0.800199);
  \draw[color=red] (30:\mlabeldist) node[anchor=south west] {$e_2^{\perp}$};
  \draw[color=red] (150:\mlabeldist) node[anchor=south east] {$e_3^{\perp}$};
  \draw[color=red] (-90:\mlabeldist) node[anchor=north] {$e_1^{\perp}$};
  \draw[thin,color=red] (0.608633,0.351395) -- (0.559944,0.435727) -- (0.646430,0.485660);
  \draw[thin,color=red, rotate=60] (0.608633,0.351395) -- (0.559944,0.435727) -- (0.646430,0.485660);
  \draw[thin,color=red, rotate=120] (0.608633,0.351395) -- (0.559944,0.435727) -- (0.646430,0.485660);
  \draw[thin,color=red, rotate=180] (0.608633,0.351395) -- (0.559944,0.435727) -- (0.646430,0.485660);
  \draw[thin,color=red, rotate=-120] (0.608633,0.351395) -- (0.559944,0.435727) -- (0.646430,0.485660);
  \draw[thin,color=red, rotate=-60] (0.608633,0.351395) -- (0.559944,0.435727) -- (0.646430,0.485660);
  \draw[thin,dotted,color=blue] (30:0.800199) -- (V123);
  \draw[thin,dotted,color=blue] (90:0.800199) -- (V213);
  \draw[thin,dotted,color=blue] (150:0.800199) -- (V231);
  \draw[thin,dotted,color=blue] (-150:0.800199) -- (V321);
  \draw[thin,dotted,color=blue] (-90:0.800199) -- (V312);
  \draw[thin,dotted,color=blue] (-30:0.800199) -- (V132);
  \draw[fill=blue] (V123) circle (1.1pt);
  \draw[color=blue] (V123) node[anchor=south] {$v_{123}$};
  \draw[fill=blue] (V213) circle (1.1pt);
  \draw[color=blue] (V213) node[anchor=south] {$v_{213}$};
  \draw[fill=blue] (V231) circle (1.1pt);
  \draw[color=blue] (V231) node[anchor=south] {$v_{231}$};
  \draw[fill=blue] (V321) circle (1.1pt);
  \draw[color=blue] (V321) node[anchor=north] {$v_{321}$};
  \draw[fill=blue] (V312) circle (1.1pt);
  \draw[color=blue] (V312) node[anchor=north] {$v_{312}$};
  \draw[fill=blue] (V132) circle (1.1pt);
  \draw[color=blue] (V132) node[anchor=north] {$v_{132}$};
\end{tikzpicture}
  \caption{Core polygon when $0 \leq t < \frac{1}{2} - \frac{1}{p}$}
  \label{fig:small_phase}
\end{figure}
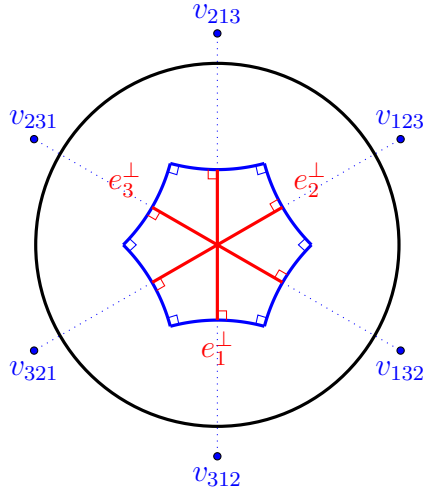
\begin{figure}[h!]
  \centering
  \begin{tikzpicture}[scale=1.25]
  \def\outvrad{2.25}
  \def\hexcentrad{3.863400}
  \coordinate (V321) at (-150:1.931852);
  \coordinate (V132) at (-30:1.931852);
  \coordinate (V213) at (90:1.931852);
  \coordinate (V123) at (30:\outvrad);
  \coordinate (V231) at (150:\outvrad);
  \coordinate (V312) at (-90:\outvrad);
  \coordinate (center321) at (-150:\hexcentrad);
  \coordinate (center312) at (-90:\hexcentrad);
  \coordinate (center132) at (-30:\hexcentrad);
  \coordinate (center123) at (30:\hexcentrad);
  \coordinate (center213) at (90:\hexcentrad);
  \coordinate (center231) at (150:\hexcentrad);
  \draw[very thick] (0,0) circle (1.93185);
  \draw [shift={(center123)},very thick,color=blue]  plot[domain=3.141593:4.188790,variable=\t]({3.345715*cos(\t r)},{3.345715*sin(\t r)});
  \draw [shift={(center231)},very thick,color=blue,rotate=120]  plot[domain=3.141593:4.188790,variable=\t]({3.345715*cos(\t r)},{3.345715*sin(\t r)});
  \draw [shift={(center312)},very thick,color=blue,rotate=-120]  plot[domain=3.141593:4.188790,variable=\t]({3.345715*cos(\t r)},{3.345715*sin(\t r)});
  \draw [very thick,color=red] (V321)-- (30:0.517685);
  \draw [very thick,color=red] (V132)-- (150:0.517685);
  \draw [very thick,color=red] (V213)-- (-90:0.517685);
  \draw[color=red] (30:0.517685) node[anchor=south west] {$e_2^{\perp}$};
  \draw[color=red] (150:0.517685) node[anchor=south east] {$e_3^{\perp}$};
  \draw[color=red] (-90:0.517685) node[anchor=north] {$e_1^{\perp}$};
  \draw[thin,color=red] (0.363986,0.210148) -- (0.315297,0.294481) -- (0.400862,0.343881);
  \draw[thin,color=red,rotate=120] (0.363986,0.210148) -- (0.315297,0.294481) -- (0.400862,0.343881);
  \draw[thin,color=red,rotate=-120] (0.363986,0.210148) -- (0.315297,0.294481) -- (0.400862,0.343881);
  \draw [thin,dotted,color=blue] (30:0.517685) -- (V123);
  \draw [thin,dotted,color=blue] (150:0.517685) -- (V231);
  \draw [thin,dotted,color=blue] (-90:0.517685) -- (V312);
  \draw [fill=blue] (V123) circle (1.1pt);
  \draw[color=blue] (V123) node[anchor=south] {$v_{123}$};
  \draw [fill=blue] (V231) circle (1.1pt);
  \draw[color=blue] (V231) node[anchor=south] {$v_{231}$};
  \draw [fill=blue] (V312) circle (1.1pt);
  \draw[color=blue] (V312) node[anchor=north] {$v_{312}$};
  \draw [fill=blue] (V321) circle (1.1pt);
  \draw[color=blue] (V321) node[anchor=north east] {$v_{321}$};
  \draw [fill=blue] (V132) circle (1.1pt);
  \draw[color=blue] (V132) node[anchor=north west] {$v_{132}$};
  \draw [fill=blue] (V213) circle (1.1pt);
  \draw[color=blue] (V213) node[anchor=south] {$v_{213}$};
\end{tikzpicture}
  \caption{Core polygon when $t= \frac{1}{2} - \frac{1}{p}$}
  \label{fig:critical_phase}
\end{figure}
\begin{figure}[h!]
  \centering
  \begin{tikzpicture}[scale=1.25]
  \def\outvrad{2.25}
  \def\hexcentrad{4.396151}
  \coordinate (V321) at (-150:1.140878);
  \coordinate (V132) at (-30:1.140878);
  \coordinate (V213) at (90:1.140878);
  \coordinate (V123) at (30:\outvrad);
  \coordinate (V231) at (150:\outvrad);
  \coordinate (V312) at (-90:\outvrad);
  \coordinate (center321) at (-150:\hexcentrad);
  \coordinate (center312) at (-90:\hexcentrad);
  \coordinate (center132) at (-30:\hexcentrad);
  \coordinate (center123) at (30:\hexcentrad);
  \coordinate (center213) at (90:\hexcentrad);
  \coordinate (center231) at (150:\hexcentrad);
  \draw [very thick] (0,0) circle (1.93185);
  \draw[shift={(center123)},very thick,blue]  plot[domain=3.412454:3.917929,variable=\t]({3.951237*cos(\t r)},{3.951237*sin(\t r)});
  \draw[shift={(center231)},very thick,blue,rotate=120]  plot[domain=3.412454:3.917929,variable=\t]({3.951237*cos(\t r)},{3.951237*sin(\t r)});
  \draw[shift={(center312)},very thick,blue,rotate=-120]  plot[domain=3.412454:3.917929,variable=\t]({3.951237*cos(\t r)},{3.951237*sin(\t r)});
  \draw[very thick,red] (90:1.140878) -- (-90:0.444914);
  \draw[very thick,red,rotate=120] (90:1.140878) -- (-90:0.444914);
  \draw[very thick,red,rotate=-120] (90:1.140878) -- (-90:0.444914);
  \draw[color=red] (30:0.444914) node[anchor=south west] {$e_2^{\perp}$};
  \draw[color=red] (150:0.444914) node[anchor=south east] {$e_3^{\perp}$};
  \draw[color=red] (-90:0.444914) node[anchor=north] {$e_1^{\perp}$};
  \draw[thin,red] (0.097379,-0.446114) -- (0.097379,-0.347527)-- (0,-0.34752732896440286);
  \draw[thin,red,rotate=120] (0.097379,-0.446114) -- (0.097379,-0.347527)-- (0,-0.347527);
  \draw[thin,red,rotate=-120] (0.097379,-0.446114) -- (0.097379,-0.347527)-- (0,-0.347527);
  \draw [thin,dotted,blue] (30:0.444914)-- (V123);
  \draw [thin,dotted,blue] (150:0.444914)-- (V231);
  \draw [thin,dotted,blue] (-90:0.444914)-- (V312);
  \draw[fill=blue] (V123) circle (1.1pt);
  \draw[color=blue] (V123) node[anchor=south] {$v_{123}$};
  \draw[fill=blue] (V231) circle (1.1pt);
  \draw[color=blue] (V231) node[anchor=south] {$v_{231}$};
  \draw[fill=blue] (V312) circle (1.1pt);
  \draw[color=blue] (V312) node[anchor=north] {$v_{312}$};
  \draw[fill=blue] (V321) circle (1.1pt);
  \draw[color=blue] (V321) node[anchor=north] {$v_{321}$};
  \draw[fill=blue] (V213) circle (1.1pt);
  \draw[color=blue] (V213) node[anchor=south] {$v_{213}$};
  \draw[fill=blue] (V132) circle (1.1pt);
  \draw[color=blue] (V132) node[anchor=north] {$v_{132}$};
\end{tikzpicture}
  \caption{Core polygon when $t > \frac{1}{2} - \frac{1}{p}$}
  \label{fig:large_phase}
\end{figure}

The following are readily checked:

\begin{prop}[Proposition 2.13(3) in \cite{DFP}]
  $v_{ijk} \perp v_{jik}$ and $v_{ijk} \perp v_{ikj}$.
\end{prop}
\begin{prop}
  $e_i \perp v_{jik}$ and $e_i \perp v_{kij}$.
\end{prop}

For the hybrid construction, we use the projective subspaces (considered as projective subspaces of $\mbf{H}_{\C}^{2}$) corresponding to $e_1^{\perp}$ and $v_{312}^{\perp}$; this is sufficient as the remaining subspaces are obtained by successive applications of $J$. In homogeneous coordinates, one sees that
\begin{align*}
  e_1^{\perp} &= \{[z,\varphi z / \alpha - \varphi^2 , 1]^{T} \; : \; z \in \C\} \quad \text{ and } \\
  v_{312}^{\perp} &= \{[z, i \bar{\eta} \bar{\varphi}, 1]^{T} \; : \; z \in \C\}.
\end{align*}

Let $\Lambda_{ijk} \leq \tilde{\Gamma}(p,t)$ be the subgroup stabilizing $v_{ijk}^{\perp}$ and let $\Lambda_i$ the subgroup stabilizing $e_i^{\perp}$. These groups are naturally identified with subgroups of $\PU(1,1)$, and so we let $\Gamma_{ijk}$ and $\Gamma_i$ be lifts of these groups (respectively) into $\SU(1,1)$. 

\begin{prop}\label{prop:Gamma312}
  $\Gamma_{312}$ is a lattice in $\SU(1,1)$. It is cocompact for all non-critical phase shift values.
\end{prop}

\begin{proof}
  $v_{312}$ is a positive eigenvector for both $R_1$ and $R_3 J$, hence they both stabilize $v_{312}^{\perp}$. The action of these elements on $v_{312}^{\perp}$ can be seen below:
  \begin{align*}
    R_1: \begin{bmatrix}z \\ i \bar{\eta} \bar{\varphi} \\ 1 \end{bmatrix} & \mapsto \begin{bmatrix} \eta^2 z + \bar{\varphi}^2 - i \eta \varphi \\ i \bar{\eta} \bar{\varphi} \\ 1 \end{bmatrix} \\
    R_3 J: \begin{bmatrix} z \\ i \bar{\eta} \bar{\varphi} \\ 1 \end{bmatrix} & \mapsto \begin{bmatrix} \dfrac{i \bar{\eta} \bar{\varphi}}{z} \\  i \bar{\eta} \bar{\varphi} \\ 1 \end{bmatrix}
  \end{align*}
  Let $A$ and $B$ be the following elements in $\SU(1,1)$ corresponding to the actions of $R_1$ and $R_3 J$ on $v_{312}^{\perp}$, respectively,
  \begin{align*}
    A &= \frac{1}{\eta} \begin{pmatrix} \eta^2 & \bar{\varphi}^{2} - i \eta \varphi \\ 0 & 1 \end{pmatrix},
    &
    B &= \frac{1}{\sqrt{- i \bar{\eta} \bar{\varphi}}} \begin{pmatrix} 0 & i\bar{\eta}\bar{\varphi} \\ 1 & 0 \end{pmatrix}.
  \end{align*}
  One then sees that
  \begin{align*}
    |\tr(A)| &= \left|1 + e^{i 2\pi/p}\right|, \\
    |\tr(B)| &= 0, \\
    |\tr(A^{-1} B)| &= \left|1 + e^{i \pi (t -1/2 + 1/p)}\right|.
  \end{align*}
  All of these values are less than or equal to $2$ for all admissible $p$ and $t$, so neither $A$ nor $B$ is loxodromic and thus they generate the orientation-preserving subgroup of a Fuchsian triangle group of finite covolume. It follows that $\Gamma_{312}$ is a lattice in $\PU(1,1)$. By computing orders of these elements for admissible $(p,t)$, one obtains Table \ref{table:Gamma312} showing the corresponding triangle groups, and arithmeticity/non-arithmeticity (A/NA) of each can be checked by comparing with the main theorem of \cite{TakeuchiTriangle}.
\end{proof}

\begin{table}[h]
  \centering
  \begin{tabular}{|c|c|c||c|c|c|}
    \hline
    $(p,t)$      & $\triangle(x,y,z)$      & A/NA & $(p,t)$       & $\triangle(x,y,z)$      & A/NA \\ \hline
    \hline
    $(3, 0)$     & $\triangle(2,3,12)$     & A    & $(4, 0)$      & $\triangle(2,4,8)$      & A    \\ \hline
    $(3, 1/30)$  & $\triangle(2,3,15)$     & NA   & $(4, 1/12)$   & $\triangle(2,4,12)$     & A    \\ \hline
    $(3, 1/18)$  & $\triangle(2,3,18)$     & A    & $(4, 3/20)$   & $\triangle(2,4,20)$     & NA   \\ \hline 
    $(3, 1/12)$  & $\triangle(2,3,24)$     & A    & $(4, 1/4)$    & $\triangle(2,4,\infty)$ & A    \\ \hline
    $(3, 5/42)$  & $\triangle(2,3,42)$     & NA   & $(4, 5/12)$   & $\triangle(2,4,12)$     & A    \\ \hline
    $(3, 1/6)$   & $\triangle(2,3,\infty)$ & A    & $(5, 1/10)$   & $\triangle(2,5,10)$     & A    \\ \hline
    $(3, 7/30)$  & $\triangle(2,3,30)$     & A    & $(5, 1/5)$    & $\triangle(2,5,20)$     & A    \\ \hline
    $(3, 1/3)$   & $\triangle(2,3,12)$     & A    & $(5, 11/30)$  & $\triangle(2,5,30)$     & A    \\ \hline
                 &                         &      & $(5, 7/10)$   & $\triangle(2,5,5)$      & A    \\ \hline
  \end{tabular}
  \vskip0.5em
  \captionof{table}{Properties of $\Gamma_{312}$}
  \label{table:Gamma312}
\end{table}

\begin{prop}\label{prop:Gamma1}
  $\Gamma_1$ is a lattice in $\SU(1,1)$. It is cocompact for all non-critical phase shift values.
\end{prop}

\begin{proof}
  $J^{-1} R_1 R_2$ and $J R_1 R_3$ both stabilize $e_1^{\perp}$:
  \begin{align*}
    J^{-1} R_1 R_2: \begin{bmatrix}z \\ \dfrac{\varphi}{\alpha}(z) - \varphi^2 \\ 1 \end{bmatrix} &\mapsto \begin{bmatrix}\dfrac{\alpha \eta^2 \varphi^3 z + \alpha \varphi - i \alpha \eta \varphi^4}{\eta^2 \varphi^2 z + -i \alpha \eta - i \eta \varphi^3} \\ \dfrac{\varphi}{\alpha} \!\left(\dfrac{\alpha \eta^2 \varphi^3 z + \alpha \varphi - i \alpha \eta \varphi^4}{\eta^2 \varphi^2 z + -i \alpha \eta - i \eta \varphi^3}\right) - \varphi^2 \\ 1 \end{bmatrix} \\
    J R_1 R_3: \begin{bmatrix} z \\ \dfrac{\varphi}{\alpha}(z) - \varphi^2 \\ 1 \end{bmatrix} & \mapsto \begin{bmatrix}\dfrac{(i \eta \varphi^3 + i \alpha \eta)z - i \eta \alpha \varphi^4 - \alpha  \eta^2 \varphi}{-\varphi^2 z + \alpha \varphi^3} \\ \dfrac{\varphi}{\alpha}\!\left(\dfrac{(i \eta \varphi^3 + i \alpha \eta)z - i \eta \alpha \varphi^4 - \alpha  \eta^2 \varphi}{-\varphi^2 z + \alpha \varphi^3}\right) \\ 1 \end{bmatrix}
  \end{align*}
  Let $A$ and $B$ be the following elements in $\SU(1,1)$ corresponding to the actions of $J^{-1} R_1 R_2$ and $J R_1 R_3$ on $e_1^{\perp}$, respectively.
  \begin{align*}
    A = \frac{1}{\alpha \sqrt{-i \eta  \varphi^3}}\begin{pmatrix}\alpha \eta^2 \varphi^3 & \alpha \varphi - i \alpha \eta \varphi^4 \\ \eta^2 \varphi^2 & -i \alpha \eta - i \eta \varphi^3 \end{pmatrix},
    \qquad 
    B = \frac{1}{\alpha  \sqrt{i \eta^3 \varphi^3}}\begin{pmatrix} i \eta \varphi^3 + i \alpha \eta & -i \eta \alpha \varphi^4 - \alpha \eta^2 \varphi \\ -\varphi^2 & \alpha  \varphi^3 \end{pmatrix}.\\
  \end{align*}
  One then sees that
  \begin{align*}
    |\tr(A)| &= \left|1 + e^{\pi i(t + 1/2 - 1/p)} \right|, \\
    |\tr(B)| &= \left|1 + e^{\pi i(t - 1/2 + 1/p)} \right|, \\
    |\tr(A B)| &= \left|-1 + e^{6 \pi i / p}\right|.
  \end{align*}
  All of these values are less than or equal to $2$ for admissible values of $p$ and $t$, so neither $A$ nor $B$ is loxodromic and thus they generate the orientation-preserving subgroup of a Fuchsian triangle group of finite covolume. It follows that $\Gamma_{312}$ is a lattice in $\PU(1,1)$. By computing orders of these elements for admissible $(p,t)$, one obtains Table \ref{table:Gamma312} showing the corresponding triangle groups, and arithmeticity/non-arithmeticity (A/NA) of each can be checked by comparing with the main theorem of \cite{TakeuchiTriangle}.
\end{proof}

\begin{table}[h]
  \centering
  \begin{tabular}{|c|c|c||c|c|c|}
    \hline
    $(p,t)$      & $\triangle(x,y,z)$      & A/NA & $(p,t)$      & $\triangle(x,y,z)$      & A/NA \\ \hline\hline
    $(3, 0)$     & $\triangle(2,12,12)$    & A    & $(4, 0)$     & $\triangle(4,8,8)$      & A    \\ \hline
    $(3, 1/30)$  & $\triangle(2,10,15)$    & NA   & $(4, 1/12)$  & $\triangle(4,6,12)$     & NA   \\ \hline
    $(3, 1/18)$  & $\triangle(2,6,18)$     & A    & $(4, 3/20)$  & $\triangle(4,5,20)$     & NA   \\ \hline 
    $(3, 1/12)$  & $\triangle(2,8,24)$     & NA   & $(4, 1/4)$   & $\triangle(4,4,\infty)$ & A    \\ \hline
    $(3, 5/42)$  & $\triangle(2,7,42)$     & NA   & $(4, 5/12)$  & $\triangle(3,4,12)$     & A    \\ \hline
    $(3, 1/6)$   & $\triangle(2,6,\infty)$ & A    & $(5, 1/10)$ & $\triangle(5,10,10)$     & A    \\ \hline
    $(3, 7/30)$  & $\triangle(2,5,30)$     & A    & $(5, 1/5)$ & $\triangle(4,10,20)$      & NA   \\ \hline
    $(3, 1/3)$   & $\triangle(2,4,12)$     & A    & $(5, 11/30)$ & $\triangle(3,10,30)$    & A    \\ \hline
    &                         &      & $(5, 7/10)$ & $\triangle(2,5,10)$      & A    \\ \hline
  \end{tabular}
  \vskip0.5em
  \captionof{table}{Properties of $\Gamma_{1}$}
  \label{table:Gamma1}
\end{table}

\begin{lem}
  Let $K = \langle J R_1 R_3,\, J R_2 R_1,\, J R_3 R_2 \rangle$. For all admissible $p,t$, $K$ is normal in $\tilde{\Gamma}(p,t)$.
\end{lem}

\begin{proof}
  For indices $i,j,k$ with $k=i+1 \pmod 3$ and $j=i-1 \pmod 3$, the following equations are readily checked:
  \begin{align*}
    R_i (J R_i R_j) R_i^{-1} &= J R_i R_j, 
    & R_k (J R_i R_j) R_k^{-1} &= (J R_i R_j) (J R_j R_k) (J R_i R_j)^{-1},\\
    R_j (J R_i R_j) R_j^{-1} &= J R_k R_i, 
    & J (J R_i R_j) J^{-1} &= J R_k R_i.
  \end{align*}
\end{proof}

\begin{lem}
  For each admissible pair $(p,t)$, the group $K$ (as in the previous lemma) has finite index in $\tilde{\Gamma}(p,t)$.
\end{lem}

\begin{proof}
  $\tilde{\Gamma}(p,t)$ is a quotient of the finitely-presented group
  \begin{align*}
    \left\langle J,\, R_1,\, R_2,\, R_3\; | \; J^3 = R_i^p = \Id,\; R_i R_{i+1} R_i = R_{i+1} R_i R_{i+1},\; R_{i+1} = J R_i J^{-1} \right\rangle 
  \end{align*}
  where $i=1,2,3$ (and indices are taken modulo 3). Let $\mathcal{X}_{\Gamma}$ be some set of additional relations so that $\tilde{\Gamma}(p,t)$ has the presentation
  \begin{align*}
  \left\langle J,\, R_1,\, R_2,\, R_3\; | \; \mathcal{X}_{\Gamma},\;,J^3 = R_i^p = \Id,\; R_i R_{i+1} R_i = R_{i+1} R_i R_{i+1},\; R_{i+1} = J R_i J^{-1} \right\rangle.
  \end{align*}
  As $K$ is normal, we examine the quotient $\tilde{\Gamma}(p,t)/K$ with presentation
  \begin{align*}
    \left\langle J,\, R_1,\, R_2,\, R_3 \; | \; \mathcal{X}_{\Gamma},\; J^3=R_i^p=J R_{i+1} R_i = \Id,\, R_i R_{i+1} R_i = R_{i+1} R_i R_{i+1},\, R_{i+1} = J R_i J^{-1}\right\rangle.
  \end{align*}
  where, again, $i=1,2,3$ and the indices are taken modulo $3$. Because $\tilde{\Gamma}(p,t)$ is generated by $R_1$ and $J$, many of the relations are superfluous, so the presentation for $\tilde{\Gamma}(p,t)/K$ simplifies a bit to
  \begin{align*}
    \langle J,\, R_1,\, R_2 \; | \; \mathcal{X}_{\Gamma},\;J^3=R_1^p=J R_2 R_1=\Id,\;R_2 = J R_1 J^{-1},\; R_1 R_2 R_1 = R_2 R_1 R_2 \rangle.
  \end{align*}
  The relation $J R_2 R_1 = \Id$ also makes the braid relation $R_1 R_2 R_1 = R_2 R_1 R_2$ superfluous, and so the presentation simplifies more to
  \begin{align*}
    \tilde{\Gamma}(p,t)/K = \left\langle J,\, R_1 \; | \; \mathcal{X}_{\Gamma},\; R_1^p=J^3=(J^{-1} R_1)^2=\Id \right\rangle.
  \end{align*}
  In this way, one sees that $\tilde{\Gamma}(p,t)/K$ is a quotient of the (orientation-preserving) $(2,3,p)$-triangle group. These triangle groups are finite when $p=3,4,5$, thus $K$ has finite index in $\tilde{\Gamma}(p,t)$.
\end{proof}

\begin{thm}
  For each admissible pair $(p,t)$, the hybrid $H(\Gamma_1, \Gamma_{312})$ has finite index in $\tilde{\Gamma}(p,t)$.
\end{thm}

\begin{proof}
  From the previous lemma, it suffices to show that the hybrid
  \begin{align*}
    H := H(\Gamma_1, \Gamma_{312}) = \langle \Lambda_1, \Lambda_{312} \rangle
  \end{align*}
  contains $K$. Indeed, $H$ contains the subgroup $\langle J^{-1}R_1R_2, J R_1 R_3, R_1, R_3 J \rangle$ by Propositions \ref{prop:Gamma312} and \ref{prop:Gamma1}, from which it immediately follows that $J R_1 R_3 \in H$. That $H$ contains the other two generators for $K$ is again a straightforward matrix computation.
  \begin{align*}
    J R_2 R_1 &= J (J^{-1} R_3 J) R_1 = (R_3 J)(R_1), \qquad \text{ and } \\
    J R_3 R_2 &= J R_3 (J^{-1} J) R_2 (J^{-1} J) = (R_1)(R_3 J).
  \end{align*}
\end{proof}

By comparing with the table on Page 418 of \cite{MaclachlanReid}, one sees that $\Gamma_1$ and $\Gamma_{312}$ are both arithmetic and non-commensurable in the case that $(p,t)=(5,11/30)$. Since $H(\Gamma_1, \Gamma_{312})$ has finite index in $\tilde{\Gamma}(5,11/30)$, it is non-arithmetic and thus

\begin{cor}
  For $(p,t)=(5,11/30)$, $H(\Gamma_1, \Gamma_{312})$ is a non-arithmetic lattice obtained by hybridizing two noncommesurable arithmetic lattices.
\end{cor}

\section{Small phase shift hybrids}

In that $\tilde{\Gamma}(p,t)$ has small phase shift, we can instead consider the hybrid with subspaces $v_{312}^{\perp}$ and $v_{321}^{\perp}$. In homogeneous coordinates, one sees that
\begin{align*}
  v_{321}^{\perp} &= \{[i \bar{\eta} \varphi, z, 1]^{T} \; : \; z \in \C\}.
\end{align*}

\begin{prop}
  $\Gamma_{321}$ is an arithmetic cocompact lattice in $\SU(1,1)$ for all small phase shift values.
\end{prop}

\begin{proof}
  $R_2$ and $J R_3^{-1}$ both stabilize $v_{321}^{\perp}$:
  \begin{align*}
    R_2: \begin{bmatrix}i \bar{\eta} \varphi \\ z \\ 1 \end{bmatrix} & \mapsto \begin{bmatrix} i \bar{\eta} \varphi \\ \eta^2 z + \varphi^2 - i \eta \bar{\varphi} \\ 1 \end{bmatrix} \\
    J R_3^{-1}: \begin{bmatrix} i \bar{\eta} \varphi \\ z \\ 1 \end{bmatrix} & \mapsto \begin{bmatrix} i \bar{\eta} \varphi \\ \dfrac{i \bar{\eta}\varphi}{z} \\ 1 \end{bmatrix}
  \end{align*}
  Let $A$ and $B$ be the following elements in $\SU(1,1)$ corresponding to the actions of $R_2$ and $J R_3^{-1}$ on $v_{321}^{\perp}$, respectively.
  \begin{align*}
    A &= \frac{1}{\eta}\begin{pmatrix} \eta^2 & \varphi^2 - i \eta \bar{\varphi} \\ 0 & 1 \end{pmatrix},
    &
    B &= \frac{1}{\sqrt{-i \bar{\eta} \varphi}}\begin{pmatrix} 0 & i \bar{\eta} \varphi \\ 1 & 0 \end{pmatrix}.
  \end{align*}
  One can check that
  \begin{align*}
    |\tr(A)| &= |e^{i \pi/p} + e^{-i \pi/p}|, \\
    |\tr(B)| &= 0,  \\
    |\tr(A^{-1} B)| &= |e^{i\pi(1/2 + 1/p - t/3)} - e^{2\pi i t/3}|.
  \end{align*}
  All of these values are less than $2$ when $p \geq 3$ and $|t| \neq \frac{1}{2} - \frac{1}{p}$ and so the elements are elliptic. Thus $\langle A, B \rangle$ is a cocompact triangle group (and therefore $\Gamma_{321}$ is a cocompact lattice). By computing orders of these elements for $(p,t)$ values in Table ~\ref{table:discrete_vals}, one obtains Table ~\ref{table:Gamma321} showing the corresponding triangle groups, and arithmeticity/non-arithmeticity (A/NA) of each can be checked by comparing with the main theorem of \cite{TakeuchiTriangle}.
\end{proof}

\begin{table}
  \centering
  \begin{tabular}{|c|c|c||c|c|c|}
    \hline
    $(p,t)$      & $\triangle(x,y,z)$  & A/NA & $(p,t)$ & $\triangle(x,y,z)$ & A/NA \\
    \hline
    $(3,0)$      & $\triangle(2,3,12)$ & A    & $(4,0)$     & $\triangle(2,4,8)$  & A  \\
    $(3,1/30)$   & $\triangle(2,3,10)$ & A    & $(4,1/12)$  & $\triangle(2,4,6)$  & A  \\
    $(3,1/18)$   & $\triangle(2,3,9)$  & A    & $(4,3/20)$  & $\triangle(2,4,5)$  & A  \\ 
    $(3,1/12)$   & $\triangle(2,3,8)$  & A    & $(5,1/10)$  & $\triangle(2,5,5)$  & A  \\
    $(3,5/42)$   & $\triangle(2,3,7)$  & A    & $(5,1/5)$   & $\triangle(2,4,5)$  & A  \\
    \hline
  \end{tabular}
  \vskip0.5em
  \captionof{table}{Properties of $\Gamma_{321}$}
  \label{table:Gamma321}
\end{table}

\begin{thm}\label{thm:mostow_hybrid}
  For $|t| < \frac{1}{2} - \frac{1}{p}$, the hybrid $H(\Gamma_{312},\Gamma_{321})$ is the full lattice $\tilde{\Gamma}(p,t)$.
\end{thm}

\begin{proof}
  The group $K=\langle R_1, R_3 J, R_2, J R_3^{-1} \rangle$ is a subgroup of $H(\Gamma_{312},\Gamma_{321})$. Since $J = (R_3 J)^{-1}(J R_3^{-1})^{-1}$, $K = \langle R_1, J \rangle = \tilde{\Gamma}(p,t)$.
\end{proof}

By comparing with the table on Page 418 of \cite{MaclachlanReid}, one sees that $\Gamma_{312}$ and $\Gamma_{321}$ are both arithmetic and noncommensurable in the cases where $(p,t) = (4,1/12)$ and $(5,1/5)$. Thus 

\begin{cor}
  $\tilde{\Gamma}(4,1/12)$ and $\tilde{\Gamma}(5,1/5)$ are non-arithmetic lattices obtained by interbreeding two noncommensurable arithmetic lattices.
\end{cor}

\bibliographystyle{alpha}
\bibliography{references}

\end{document}